\documentclass[12pt]{amsart}
\usepackage[active]{srcltx}
\usepackage{verbatim}
\usepackage{epsfig,graphicx,color,mathrsfs}
\usepackage{graphicx}
\usepackage{amsmath,amssymb,amsthm,amsfonts}
\usepackage{amssymb}
\usepackage[english]{babel}
\usepackage{epsfig,graphicx,color,mathrsfs}
\usepackage[english]{babel}
\usepackage[left=2.7cm,right=2.7cm,top=3cm,bottom=3cm]{geometry}

\newcommand{\R}{{\mathbb R}}

\numberwithin{equation}{section}
\newtheorem{theorem}{Theorem}[section]
\newtheorem{proposition}[theorem]{Proposition}
\newtheorem{lemma}[theorem]{Lemma}

\newtheorem{remark}[theorem]{Remark}

\theoremstyle{definition}

\renewcommand{\dfrac}{\displaystyle\frac}
\newcommand{\brm}{\begin{remark}\rm}
\newcommand{\erm}{\end{remark}}
\newcommand{\brms}{\begin{remark}\rm}
\newcommand{\erms}{\end{remark}}
\newcommand{\bte}{\begin{theorem}}
\newcommand{\ete}{\end{theorem}}
\newcommand{\bpr}{\begin{proposition}}
\newcommand{\epr}{\end{proposition}}
\newcommand{\ble}{\begin{lemma}}
\newcommand{\ele}{\end{lemma}}
\newcommand{\beq}{\begin{equation}}
\newcommand{\eeq}{\end{equation}}
\newcommand{\bdm}{\begin{displaymath}}
\newcommand{\edm}{\end{displaymath}}
\numberwithin{equation}{section}

\newcommand{\bos}{\begin{remark}\rm}
\newcommand{\eos}{\end{remark}}

\newcommand{\ben}{\begin{enumerate}}
\newcommand{\een}{\end{enumerate}}

\newcommand{\be}{\begin{equation}}
\newcommand{\ee}{\end{equation}}

\title[Classification of the solutions]{Classification of positive $\mathcal {D}^{1,p}(\R^N)$-solutions to the critical $p$-Laplace equation in $\mathbb{R}^N$}

\author[B.\ Sciunzi]{Berardino Sciunzi}

\thanks{\it 2010 Mathematics Subject
 Classification: 35J92,35B33,35B06}

\thanks{Dipartimento di Matematica e Informatica,
Universit\`a della Calabria,
Ponte Pietro Bucci 31B, I-87036 Arcavacata di Rende, Cosenza, Italy,
E-mail: {\em sciunzi@mat.unical.it}}

\thanks{Partially supported by ERC-2011-grant: \emph{Elliptic PDE's and symmetry of interfaces and layers for odd nonlinearities.}}

\thanks{Partially supported by PRIN-2011: {\em Variational and Topological Methods in the Study of Nonlinear Phenomena}}

\begin{document}

\begin{abstract}
We provide the classification of the positive solutions to $-\Delta_p u =u^{p^*-1}$ in  $\mathcal {D}^{1,p}(\R^N)$ in the case $2<p<N$. Since the case $1<p\leq2$ is already known this provides the complete classification for $1<p<N$.
\end{abstract}

\maketitle

\tableofcontents

\medskip
\section{Introduction}\label{introdue}
We consider in the whole space the critical problem
\begin{equation}\nonumber
\mathcal{P}^*\,:=\qquad
\begin{cases}
-\Delta_p u =u^{p^*-1}\quad \,\,\text{in }\,\,\mathbb{R}^N\\
\qquad\, u >0 \, \qquad \quad\text{in }\,\,\mathbb{R}^N \\
\qquad\, u\in \mathcal {D}^{1,p}(\mathbb{R}^N)
\end{cases}
\end{equation}
where $1< p<N$,  $p^*=\frac{Np}{N-p}$ is the critical exponent for the Sobolev embedding  and
\[
\mathcal {D}^{1,p}(\R^N) = \Big\{u\in L^{p^*}(\mathbb{R}^N) \,:\, \int_{\mathbb{R}^N}|\nabla u|^p <\infty\Big\}\,.
\]
 Let us recall that
any solution $u\in \mathcal {D}^{1,p}(\R^N)$ of  $\mathcal{P}^*$ belongs to  $ L^\infty(\mathbb{R}^N)$ as it follows by  \cite{moser,localser,trudi}.
 Consequently we have that $u$ is locally of class  $C^{1,\alpha}$ by  $C^{1,\alpha}$ estimates (see \cite{Di,mingione,Li,Te,T}).\\

\noindent We deal with  the classification of the solutions to $\mathcal{P}^*$. It is well known that such issue is crucial in many applications such as a-priori estimates,
blow up analysis and asymptotic analysis.
An explicit family of solutions to $\mathcal P^*$ is given by
\begin{equation}\label{classificationveron}
U_{\lambda,x_0}:=\left[\frac{\lambda^{\frac{1}{p-1}}(N^{\frac{1}{p}}(\frac{N-p}{p-1})^{\frac{p-1}{p}})
}{
\lambda^{\frac{p}{p-1}}+|x-x_0|^{\frac{p}{p-1}}
}\right]^{\frac{N-p}{p}}\quad \lambda>0\qquad x_0\in\mathbb{R}^N.
\end{equation}
Note that, by \cite{Ta}, it follows that the family of functions given by \eqref{classificationveron} are minimizers
to
\begin{equation}\label{SSS}
S:=\,\min_{\underset{\varphi\neq 0}{\varphi\in \mathcal {D}^{1,p}(\R^N)} }\dfrac{\int_{\mathbb{R}^N}|\nabla\varphi|^pdx}{\left(\int_{\mathbb{R}^N}\varphi^{p^*}dx\right)^{\frac{p}{p^*}}}\,.
\end{equation}
By the  classification results in \cite{ClassGV}(see also \cite{bidaut})  it follows that all the regular \emph{radial} solutions to $\mathcal{P}^*$  are given by
\eqref{classificationveron}. \\

\noindent In the semilinear case $p=2$ it has been proved in the celebrated paper  \cite{CGS} (see also \cite{chen}) that any solution to $-\Delta u=u^\frac{N+2}{N-2}$ ($N\geq 3$) is radial and hence classified by \eqref{classificationveron}. It is crucial in the proof the use of the Kelvin transform that allows to reduce to the study of  the symmetry of  solutions that have nice decaying properties at infinity. Previous results were proved in \cite{GNN2} via the \emph{Moving Plane Method} under additional conditions, see \cite{S} and  \cite{GNN}.\\

\noindent In the quasilinear case $p\neq 2$ the problem is more difficult and we have to take into account the nonlinear nature of the $p$-Laplace operator, the lack of regularity of the solutions and the fact that comparison principles are not equivalent to maximum principles in this case. Furthermore
a Kelvin type transform is not available. The first result has been recently obtained in \cite{AdvS} where it was considered the case $1<p<2$ when the nonlinearity is locally Lipschitz continuous, namely $p^*\geq2$. This requires that $\frac{2N}{N+2}\leq p<2$. The result has been extended to the case $1<p<2$ in \cite{VETOIS} exploiting a fine analysis of the behaviour of the solutions at infinity
that allows to exploit the moving plane method as developed in
 \cite{DPR,DR}(see also \cite{Szou}).\\

\noindent In this paper we completely set up the problem proving that the solutions in \eqref{classificationveron} are the only solutions to $\mathcal P^*$ (for $1<p<N$). We explicitly state the result and, for the reader's convenience, we include the singular case $1<p<2$ that is already known as remarked above.
\begin{theorem}\label{radialtheorem}
Let $1<p<N$ and let $u$ be a solution to $\mathcal P^*$. Then there exist $\lambda>0$ and $x_0\in\mathbb{R}^N$ such that:
\begin{equation}\nonumber
u(x)=
U_{\lambda,x_0}(x):=\left[\frac{\lambda^{\frac{1}{p-1}}(N^{\frac{1}{p}}(\frac{N-p}{p-1})^{\frac{p-1}{p}})
}{
\lambda^{\frac{p}{p-1}}+|x-x_0|^{\frac{p}{p-1}}
}\right]^{\frac{N-p}{p}}\,.
\end{equation}
\end{theorem}

The proof of Theorem \ref{radialtheorem} follows once we show that all the solutions are radial (and radially decreasing). This will be the core of the paper. For the reader's convenience we provide an outline of the proofs:\\

\begin{itemize}
\item[i)] In Section \ref{section2} (see Theorem \ref{cjvbmcnvbmvbmvbkghkg} ) we provide a lower bound on the decay rate of $|\nabla u|$,  namely
 we show that
$
 |\nabla u (x)|\geq     \tilde C|x|^{-\frac{N-1}{p-1}}
$
for some constant $\tilde C$,
 in $ \mathbb{R}^N\setminus\{B_{R_0}\}$ (for some $R_0>0$). This is the counterpart to the upper bound obtained in \cite{VETOIS}.
 We use here a new technique based on the study of the limiting profile at infinity and we also take advantage of the a priori estimates proved in \cite{VETOIS} .\\
 
\item[ii)] In Section \ref{section3}, see
Theorem \ref{hfgdhgdhghfghdkgfvks}, we show that the moving plane technique can be carried out  exploiting $i)$, Hardy's inequality and the weighted Poincar\'{e} type inequality obtained in \cite{DS1}.\\

\item[ii)] We conclude proving Theorem \ref{radialtheorem} exploiting Theorem \ref{hfgdhgdhghfghdkgfvks} and the classification of the radial solutions in  \cite{ClassGV}.
\end{itemize}

\section{Preliminary results and decay estimates}\label{section2}
 In the following we will frequently use standard elliptic estimates. Let us recall in particular that,
for any $  p>1 $,   there exists a positive constant $C_1$,  depending on $p$, such that  $\; \forall \, \eta, \eta' \in  \mathbb{R}^{N}$
\begin{eqnarray}\label{eq:inequalities}
[|\eta|^{p-2}\eta-|\eta'|^{p-2}\eta'][\eta- \eta'] \geq C_1 (|\eta|+|\eta'|)^{p-2}|\eta-\eta'|^2\,. \\ \nonumber
\end{eqnarray}
A key tool in our proofs is the moving plane technique. To exploit it we need the following notation.
We will study the symmetry of the solutions in the $\nu$-direction for any $\nu\in S^{N-1}$. Anyway, since the problem is invariant up to rotations, we fix
$\nu=e^1=(1,0,\ldots ,0)$ and   we set
\begin{equation}\nonumber
T_\lambda=\{x\in \mathbb{R}^N\,:\,x_1=\lambda\}
\end{equation}
\begin{equation}\nonumber
\Sigma_\lambda=\{x\in \mathbb{R}^N\,:\,x_1 <\lambda\}
\end{equation}
\begin{equation}\nonumber
x_\lambda=R_\lambda^\nu(x)=x+2(\lambda -x\cdot e^1)e^1=(2\lambda - x_1,x'),\qquad x'\in\mathbb{R}^{N-1}\,,
\end{equation}
i.e. $R_\lambda $  is the reflection trough the hyperplane $T_\lambda$. Furthermore we set
\begin{equation}\nonumber
u_{\lambda }(x)=u(x_\lambda) \,.
\end{equation}
Finally we define
\begin{equation}\nonumber
\Lambda=\{\lambda\in \mathbb{R}\,:\, u\leq u_{\mu}\,\, \text{in}\,\, \Sigma_\mu,\,\, \forall \mu\leq\lambda \}.
\end{equation}
 and, if $\Lambda \neq \emptyset$, we set
\begin{equation}\label{eq:suplambda}
 \lambda _0  =\sup \Lambda.
\end{equation}
In all the paper, we use the notation
$
B_R=B_R(0)
$
to indicate the ball of radius $R$ centered at the origin.\\

Now we state a simple result, that will be used afterwards, regarding the uniqueness up to multipliers of $p$-harmonic maps in $\mathbb{R}^N\setminus\{0\}$, under suitable conditions at zero and at infinity. For the reader's convenience we add a very simple proof.

\begin{theorem}\label{thmarmonic}
Let $v\in C^{1,\alpha}_{loc}(\mathbb{R}^N\setminus\{0\})$ be $p$-harmonic in $\mathbb{R}^N\setminus\{0\}$ and assume that
\begin{equation}\label{limarm}
\underset{|x|\rightarrow \infty}{\lim}\, v(x)=0\qquad\text{and}\qquad\underset{|x|\rightarrow 0}{\lim}\, v(x)=\infty\,.
\end{equation}
Then $v$ is the fundamental solution, namely:
\begin{equation}
v(x)\,=\, \frac{\alpha}{|x|^{\frac{N-p}{p-1}}}\qquad \text{for some}\,\,\alpha\in\mathbb{R}^+\,.
\end{equation}
\end{theorem}
\begin{proof}
We start showing that $v$ is radial by exploiting the moving plane technique. To do this it is convenient to fix the direction $\nu=e_1$ and set
\[
w_\lambda\,:=\,\Big( v(x)-v_\lambda(x)-\varepsilon\Big)^+\,\chi_{\Sigma_\lambda}
\]
for $\lambda<0$ and $\varepsilon>0$ (small). By \eqref{limarm} it follows that
\[
supp \,w_\lambda \subset\subset \Sigma_\lambda\setminus 0_\lambda
\]
where $0_\lambda = (-2\lambda,x')$ is the reflected point of the origin. Consequently, since $v$ and $v_\lambda$ are $p$-harmonic
in $\mathbb{R}^N\setminus\{0\,,\,0_\lambda\}$, exploiting also \eqref{eq:inequalities},
we can infer that
\begin{equation}\nonumber
\begin{split}
C_1&\int_{\Sigma_\lambda\cap supp\, w_\lambda }\Big( |\nabla v|+ |\nabla v_{\lambda}|\Big)^{p-2}|\nabla (v-v_{\lambda})|^2 dx\\
&\leq\int_{\Sigma_\lambda\cap supp\, w_\lambda }\langle |\nabla v|^{p-2}\nabla v- |\nabla v_{\lambda}|^{p-2}\nabla v_{\lambda},\nabla (v-v_{\lambda})\rangle dx\\
&=\int_{\Sigma_\lambda }\langle |\nabla v|^{p-2}\nabla v- |\nabla v_{\lambda}|^{p-2}\nabla v_{\lambda},\nabla w_\lambda\rangle dx\\
&=0
\end{split}
\end{equation}
showing  that $w_\lambda=0$ in $\Sigma_\lambda$, namely $v\leq v_\lambda+\varepsilon$ in $\Sigma_\lambda$. Since $\lambda$ and $\varepsilon>0$ can be arbitrary chosen, we deduce that
\[
v\leq v_\lambda\qquad \text{in}\,\,\Sigma_\lambda\qquad \text{for any}\,\, \lambda<0\,.
\]
Repeating the argument in the $(-e^1)$-direction we deduce that $v$ is symmetric in the  $e^1$-direction and monotone nondecreasing in $e^1$-direction in $\Sigma_0$. It is now easy to observe that the same procedure can be performed in any direction $\nu\in S^{N-1}$ to deduce that $u$ is radial and radially nonincreasing. We use the notation $$ u=u(r)\,.$$
This allows to exploit the strong maximum principle (see \cite{V}) and get that $u'(r)\neq 0$ for $r>0$. In particular we have that
 \[
u'(r)<0\qquad\text{for}\,\,r>0\,.
\]
 This follows by the strong maximum principle applied to the derivatives of $u$,  see e.g. \cite{PSB} recalling that the $p$-Laplace operator is no more degenerate in the set $\{\nabla u\neq 0\}$.
By standard regularity theory, since  we know that $\{\nabla u=0\}=\emptyset$, it follows that $u\in C^{2,\alpha}_{loc}(\mathbb{R}^N\setminus\{0\})$ and  $(|u'|^{p-2}u'r^{N-1})'=0$. Equivalently, since $u'(r)<0$ for $r>0$, we have that $((-u')^{p-1}r^{N-1})'=0$ and therefore
\[
(-u')^{p-1}r^{N-1}=c\,.
\]
The proof of the result follows now
integrating and exploiting \eqref{limarm}.
\end{proof}
We recall now the result by J. V\'{e}tois. It has been showed in fact in \cite[Theorem 1.1]{VETOIS} that, under our assumptions (some more general estimates are also considered in \cite{VETOIS}), it holds that
\begin{equation}\label{vet}
\bar c(1+|x|^{\frac{N-p}{p-1}})^{-1}\leq u(x)\leq \bar C(1+|x|^{\frac{N-p}{p-1}})^{-1}\quad\text{and}\quad |\nabla u (x)|\leq \bar C(1+|x|^{\frac{N-1}{p-1}})^{-1}\,.
\end{equation}
In this section we provide the corresponding lower bound for the decay rate of $|\nabla u|$. Namely we have the following:
\begin{theorem}\label{cjvbmcnvbmvbmvbkghkg}
Let $1<p<N$ and let $u$ be a solution to  $\mathcal{P}^*$. Then there exist a radius $R_0>0$ and a constant $\tilde C>0$ such that
\begin{equation}\label{sciunzi}
 |\nabla u (x)|\geq     \frac{\tilde C}{|x|^{\frac{N-1}{p-1}}}  \qquad{in }\quad \mathbb{R}^N\setminus\{B_{R_0}\}\,.
\end{equation}
\end{theorem}
\begin{proof}
Assume by contradiction that there exist sequences of radii $R_n$  and points $x_n$ with $R_n$ tending to infinity as $n$ tends to infinity and
$|x_n|=R_n$, such that
\begin{equation}\label{contradiction}
 |\nabla u (x_n)|\leq     \frac{\theta_n}{|R_n|^{\frac{N-1}{p-1}}}  \qquad{with }\quad \theta_n\underset{n\rightarrow \infty}{\longrightarrow}\,0\,.
\end{equation}
For $0<a<A$ fixed we set
\begin{equation}\nonumber
 w_{R_n}(x)\,:=\, R_n^{\frac{N-p}{p-1}}u(R_n\,x)\,.
\end{equation}
By \eqref{vet} (relabeling the constants and for $n$ large), it follows that
\begin{equation}\nonumber
\frac{\bar c}{A^{\frac{N-p}{p-1}}}\leq
 w_{R_n}(x)\leq \frac{\bar C}{a^{\frac{N-p}{p-1}}}
 \qquad\text{in }\quad \overline{B_A\setminus B_a}
\end{equation}
and in particular
\begin{equation}\label{jthvuvidasdasfasadf}
\begin{split}
&w_{R_n}\leq\frac{\bar C}{A^{\frac{N-p}{p-1}}}\qquad\text{on}\quad\partial B_{A}\\
&w_{R_n}\geq\frac{\bar c}{a^{\frac{N-p}{p-1}}}\qquad\text{on}\quad\partial B_{a}\,.
\end{split}
\end{equation}
Therefore, for $a,A$ fixed, it follows that $w_{R_n}$ is uniformly bounded in $L^\infty(B_A\setminus B_a)$ and hence, by  \cite{Di,mingione,Li,Te,T}, it is also uniformly bounded in $C^{1,\alpha}(K)$, $0<\alpha<1$, for any compact set $K\subset B_A\setminus B_a$. We agree that $a,A$ are redefined  so that the $C^{1,\alpha}$ estimates holds in  the closure of $B_A\setminus B_a$.
Up to subsequences we have that
\begin{equation}\label{nvbcbmvvmxmvbmvbxm}
 w_{R_n}(x)\overset{C^{1,\alpha'}}{\rightarrow}w_{a,A}
 \qquad\text{in}\quad B_A\setminus B_a\,.
\end{equation}
for $0<\alpha'<\alpha$.
Exploiting the fact that
\[
-\Delta_p w_{R_n}\,=\,\frac{1}{R_n^{\frac{p}{p-1}}}\, w_{R_n}^{p^*-1}
\]
we deduce that
\begin{equation}\label{dfbkvkkxvb}
-\Delta_p w_{a,A}\,=\,0 \qquad\text{in}\quad B_A\setminus \overline{B_a}\,.
\end{equation}
Now, for $j\in\mathbb{N}$, we take
\[
a_j\,=\, \frac{1}{j}\qquad\text{and}\qquad A_j=j
\]
and we construct $w_{a_j,A_j}$ as above.
 Letting $j$ tends to infinity and performing a standard diagonal process, we construct a limiting profile $w_\infty$ so that
\[
w_\infty\,\equiv\,w_{a_j,A_j}\qquad\text{in}\quad B_{A_j}\setminus B_{a_j}\,.
\]
In particular, by \eqref{dfbkvkkxvb}, it follows that
\begin{equation}\label{dfbkvkklrjbyypb xvb}
-\Delta_p w_{\infty}\,=\,0 \qquad\text{in}\quad \mathbb{R}^N\setminus\{0\}\,.
\end{equation}
By \eqref{jthvuvidasdasfasadf}, that holds with $a=a_j$ and $A=A_j$, it follows that
 \begin{equation}\nonumber
\underset{|x|\rightarrow \infty}{\lim}\, w_{\infty}(x)=0\qquad\text{and}\qquad\underset{|x|\rightarrow 0}{\lim}\, w_{\infty}(x)=\infty
\end{equation}
so that Theorem \ref{thmarmonic} can be applied to deduce that
\begin{equation}\nonumber
w_{\infty}(x)\,=\, \frac{\alpha}{|x|^{\frac{N-p}{p-1}}}\qquad \text{for some}\,\,\alpha\in\mathbb{R}^+\,.
\end{equation}
Let now $x_n$ be as in \eqref{contradiction} and set
\[
y_n=\frac{x_n}{R_n}\,.
\]
Then, by \eqref{contradiction}, it follows that $|\nabla w_{R_n}(y_n)|$ tends to zero as $n$ tends to infinity. Up to subsequences, since $|y_n|=1$, we have that $y_n$ tends to $\bar y\in\partial B_1$.  Consequently, by the uniform convergence of the gradients that follows by \eqref{nvbcbmvvmxmvbmvbxm}, we have that
\[
\nabla w_\infty (\bar y)=0.
\]
This is a contradiction since the fundamental solution has no critical points and the proof is done.
\end{proof}

\section{Radial symmetry of the solutions}\label{section3}
Now we can prove the symmetry (and monotonicity) result. Since the case $1<p<2$ is already known \cite{AdvS,VETOIS}, we only consider the case $p>2$.
We have the following:
\begin{theorem}\label{hfgdhgdhghfghdkgfvks}
Let $2<p<N$ and let $u$ be a solution to $\mathcal{P}^*$.  Then $u$ is radial and radially decreasing about some point $x_0\in\mathbb{R}^N$. Assuming up to translations that $x_0=0$ it follows that $u=u(r)$ with  $u'(r)<0$ for $r>0$.
\end{theorem}

\begin{proof}
We carry out the moving plane technique in the $e^1$-direction starting from the left. More precisely we start with:\\

\noindent \emph{Step 1.}
{ \it Let us prove that $\Lambda \neq \emptyset $ for $\lambda<0$ with $|\lambda|$ large.} \\
It is easy to see that it is
possible to take the function $ (u-u_\lambda)^+\chi_{\Sigma_\lambda}$ as a
test function in the distributional formulation of  $-\Delta_p u =u^{p^*-1}$. Recall that $u_\lambda$ also fulfills $-\Delta_p u_\lambda =u_\lambda^{p^*-1}$.
This follows by standard density arguments and exploiting e.g. \eqref{vet} but the fact that $u,u_\lambda\in L^{p^*}(\mathbb{R}^N)$ is enough as it was shown in \cite{AdvS}.
Therefore we have that:
\begin{equation}\label{eq:secondaaaa}
 \int_{\Sigma_\lambda}|\nabla u|^{p-2}\langle \nabla u,\nabla (u-u_{\lambda})^+\rangle dx=\int_{\Sigma_\lambda}u^{p^*-1}(u-u_{\lambda})^+ dx\,.
\end{equation}
Since  $u_{\lambda}$ satisfies the same equation, analogously we get
\begin{equation}\label{eq:secondaaaaA}
\int_{\Sigma_\lambda}|\nabla u_{\lambda}|^{p-2}\langle \nabla u_{\lambda},\nabla (u-u_{\lambda})^+\rangle dx=\int_{\Sigma_\lambda}{u_{\lambda}}^{p^*-1}(u-u_{\lambda})^+dx\,.
\end{equation}
Subtracting we get
\begin{equation}\nonumber
\int_{\Sigma_\lambda}\langle |\nabla u|^{p-2}\nabla u- |\nabla u_{\lambda}|^{p-2}\nabla u_{\lambda},\nabla (u-u_{\lambda})^+\rangle dx=\int_{\Sigma_\lambda}(u^{p^*-1}-{u_{\lambda}}^{p^*-1}) (u-u_{\lambda})^+ dx
\end{equation}
and, by \eqref{eq:inequalities}, we have
\begin{equation}\label{eq:madriddddd}
\int_{\Sigma_\lambda}(|\nabla u|+|\nabla u_{\lambda}|)^{p-2}|\nabla(u-u_{\lambda})^+|^2dx\leq \frac{1}{C_1}\int_{\Sigma_\lambda}(u^{p^*-1}-u_{\lambda}^{p^*-1}) (u-u_{\lambda})^+ dx\,.
\end{equation}

Using Lagrange Theorem,   the fact that $u^{p^*-1}$ is convex in $u$ and the fact that $u\geq u_\lambda$ in the support of  $(u-u_{\lambda})^+$, we obtain
\begin{equation}\label{eq:batiado}
\begin{split}
&\int_{\Sigma_\lambda}(u^{p^*-1}-u_{\lambda}^{p^*-1}) (u-u_{\lambda})^+ dx\leq (p^*-1)\int_{\Sigma_\lambda}u^{p^*-2} [(u-u_{\lambda})^+]^2 dx\\
&\leq \bar C (p^*-1)\int_{\Sigma_\lambda}\frac{1}{|x|^{(p^*-2)\frac{N-p}{p-1}}} [(u-u_{\lambda})^+]^2 dx\\
\end{split}
\end{equation}
where we also exploited \eqref{vet}. Now, in order to apply Hardy's inequality, we set
\[
s\,:=\,-\frac{N-1}{p-1}(p-2)\qquad\text{and}\qquad \beta^*\,:=\,(p^*-2)\frac{N-p}{p-1}+s-2\,.
\]
It is easy to see that $\beta^*$ is positive and that $s>2-N$ so that we are in position to apply  Hardy's inequality.
We refer the reader to the version of \cite[Lemma 2.3]{DR} where also a self contained proof is available.
In particular by
\eqref{eq:madriddddd} and \eqref{eq:batiado}, observing that $|x|>|\lambda|$ in $\Sigma_\lambda$, we get
\begin{equation}\label{eq:batiadohfhfhfh}
\begin{split}
&\int_{\Sigma_\lambda}(|\nabla u|+|\nabla u_{\lambda}|)^{p-2}|\nabla(u-u_{\lambda})^+|^2dx \\
&\leq \frac{\bar C}{C_1} (p^*-1) \frac{1}{|\lambda|^{\beta^*}}\int_{\Sigma_\lambda}|x|^{s-2}[(u-u_{\lambda})^+]^2 dx\\
&\leq \frac{\bar C}{C_1} (p^*-1) \frac{1}{|\lambda|^{\beta^*}}\left(\frac{2}{N+s-2}\right)^2
\int_{\Sigma_\lambda}|x|^s|\nabla(u-u_{\lambda})^+|^2dx\\
&\leq \frac{\bar C}{C_1\,\tilde C^{p-2}} (p^*-1) \frac{1}{|\lambda|^{\beta^*}}\left(\frac{2}{N+s-2}\right)^2
\int_{\Sigma_\lambda}|\nabla u|^{p-2}|\nabla(u-u_{\lambda})^+|^2dx\\
&\leq \frac{\bar C}{C_1\,\tilde C^{p-2}} (p^*-1) \frac{1}{|\lambda|^{\beta^*}}\left(\frac{2}{N+s-2}\right)^2
\int_{\Sigma_\lambda}(|\nabla u|+|\nabla u_{\lambda}|)^{p-2}|\nabla(u-u_{\lambda})^+|^2dx\\
\end{split}
\end{equation}
where $\tilde C$ is given by Theorem \ref{cjvbmcnvbmvbmvbkghkg}. Let us emphasize that here, since $p-2>0$, it is crucial the use of
Theorem \ref{cjvbmcnvbmvbmvbkghkg} that can be avoided on the contrary in the case $1<p<2$ since the reverse inequality is needed if $p-2<0$ .
For $|\lambda|$ large such that
\[
 \frac{\bar C}{C_1\,\tilde C^{p-2}} (p^*-1) \frac{1}{|\lambda|^{\beta^*}}\left(\frac{2}{N+s-2}\right)^2<1
\]
it follows that \eqref{eq:batiadohfhfhfh} provides a contradiction unless $(u-u_{\lambda})^+=0$. Therefore
\[
u\leq u_\lambda\qquad\text{in}\quad\Sigma_\lambda\,.
\]
Note now that, by the strong comparison principle \cite[Theorem 1.4]{DS2} we deduce that, either
$u<u_\lambda$ in $\Sigma_\lambda$, or $u=u_\lambda$ in $\Sigma_\lambda$. If $u=u_\lambda$ in $\Sigma_\lambda$ we have a symmetry hyperplane
and the proof follows performing the moving plane procedure in the other directions. If else
$u<u_\lambda$ in $\Sigma_\lambda$ we let $\lambda_0$ be defined by
\eqref{eq:suplambda} that is well defined since we showed that $\Lambda\neq\emptyset$. Since $u$ vanishes at infinity it follows that $\lambda_0$ is finite. Furthermore, by continuity, we have that
$u\leq u_{\lambda_0}$ in $\Sigma_{\lambda_0}$.\\

\noindent \emph{Step 2.}
{ \it Let us prove that $u=u_{\lambda_0}$ in $\Sigma_{\lambda_0}$.} \\
Assume by contradiction that $u$ does not coincide with $u_{\lambda_0}$ in $\Sigma_{\lambda_0}$.
 Again, by the strong comparison principle \cite[Theorem 1.4]{DS2}, we deduce that
$u<u_{\lambda_0}$ in $\Sigma_{\lambda_0}$.
Let us consider $R>0$, that later on will be fixed large, and  let us set
\[
B_R^\varepsilon\,:=\, \mathcal C (B_R)\cap \Sigma_{\lambda_0+\varepsilon}\quad\text{and}\quad
S_\delta^\varepsilon\,:=\, \left(\Sigma_{\lambda_0+\varepsilon}\setminus\Sigma_{\lambda_0-\delta}\right)\cap B_R
\quad\text{and}\quad
K_\delta\,:=\, \overline{B_R\cap \Sigma_{\lambda_0-\delta}}\,,
\]
where  $\mathcal C (\cdot)$ indicates the complementary of a set. It is clear that
\[
\Sigma_{\lambda_0+\varepsilon}=B_R^\varepsilon\cup S_\delta^\varepsilon\cup K_\delta\,.
\]
For $\delta >0$ we have that  $u_{\lambda_0}>u$ in $K_\delta$.  Therefore,
since $K_\delta$ is compact, there exists a small $\bar\varepsilon>0$ such that
 $$u_{\lambda_0+\varepsilon}>u$$ in $K_\delta$ for any $0\leq\varepsilon\leq\bar\varepsilon$. For such values of $\varepsilon$ we
argue as above  taking $(u-u_{\lambda_0+\varepsilon})^+\chi_{\Sigma_{\lambda_0+\varepsilon}}$ as test function.
In the following  the integral in $B_R^\varepsilon$ is estimated arguing exactly as in \eqref{eq:batiadohfhfhfh} and recalling that  Hardy's inequality only requires the functions to vanish at infinity. We have:
\begin{equation}\label{eq:madridddddthfgjhdfhd}
\begin{split}
&\int_{\Sigma_{\lambda_0+\varepsilon}}(|\nabla u|+|\nabla u_{\lambda_0+\varepsilon}|)^{p-2}|\nabla(u-u_{\lambda_0+\varepsilon})^+|^2dx\leq \frac{1}{C_1}\int_{\Sigma_{\lambda_0+\varepsilon}}(u^{p^*-1}-u_{\lambda_0+\varepsilon}^{p^*-1}) (u-u_{\lambda_0+\varepsilon})^+ dx\\
&=\frac{1}{C_1}\int_{B_R^\varepsilon}(u^{p^*-1}-u_{\lambda_0+\varepsilon}^{p^*-1}) (u-u_{\lambda_0+\varepsilon})^+ dx\,+
\, \frac{1}{C_1}\int_{S_\delta^\varepsilon}(u^{p^*-1}-u_{\lambda_0+\varepsilon}^{p^*-1}) (u-u_{\lambda_0+\varepsilon})^+ dx\,\\
&\leq \frac{1}{C_1}\int_{B_R^\varepsilon}(u^{p^*-1}-u_{\lambda_0+\varepsilon}^{p^*-1}) (u-u_{\lambda_0+\varepsilon})^+ dx\,+
\, \frac{(p^*-1)\|u\|_\infty^{p^*-2}}{C_1}\int_{S_\delta^\varepsilon}\big[ (u-u_{\lambda_0+\varepsilon})^+\big]^2 dx\\
&\leq \frac{\bar C}{C_1\,\tilde C^{p-2}} (p^*-1) \frac{1}{|R|^{\beta^*}}\left(\frac{2}{N+s-2}\right)^2\int_{B_R^\varepsilon}(|\nabla u|+|\nabla u_{\lambda_0+\varepsilon}|)^{p-2}|\nabla(u-u_{\lambda_0+\varepsilon})^+|^2dx\\
&+\frac{(p^*-1)\|u\|_\infty^{p^*-2}}{C_1}C_p^2(S_\delta^\varepsilon)\int_{S_\delta^\varepsilon}(|\nabla u|+|\nabla u_{\lambda_0+\varepsilon}|)^{p-2}|\nabla(u-u_{\lambda_0+\varepsilon})^+|^2dx
\end{split}
\end{equation}
where $C_p(\Omega)$ is the Poincar\'{e}  constant for the weighted  Poincar\'{e} type inequality obtained in \cite[Theorem 3.2]{DS1}. Note that the weight considered in
\cite[Theorem 3.2]{DS1} is $|\nabla u|^{p-2}$ but we use here  the fact that, since $p>2$, we have that
$|\nabla u|^{p-2}\leq \big( |\nabla u|+|\nabla u_{\lambda_0+\varepsilon}|\big)^{p-2}$.

\noindent Now we take care of the variable parameters $R,\delta,\bar\varepsilon$. First we fix $R$ large so that
\[
\frac{\bar C}{C_1\,\tilde C^{p-2}} (p^*-1) \frac{1}{|R|^{\beta^*}}\left(\frac{2}{N+s-2}\right)^2<1\,.
\]
Then, recalling that $C_p(\Omega)$ goes to zero if the Lebesgue measure of $\Omega$ goes to zero, we chose $\delta$ small so that
\[
\frac{(p^*-1)\|u\|_\infty^{p^*-2}}{C_1}C_p^2(S_\delta^0)<\frac{1}{2}\,.
\]
Finally we take $\bar\varepsilon$ as above and so that, eventually reducing it, still we have
\[
\frac{(p^*-1)\|u\|_\infty^{p^*-2}}{C_1}C_p^2(S_\delta^\varepsilon)<1
\]
for any $0\leq\varepsilon\leq\bar\varepsilon$. With this choice of parameters, by \eqref{eq:madridddddthfgjhdfhd} and reassembling the integrals, we get

\begin{equation}\nonumber
\int_{\Sigma_{\lambda_0+\varepsilon}}(|\nabla u|+|\nabla u_{\lambda_0+\varepsilon}|)^{p-2}|\nabla(u-u_{\lambda_0+\varepsilon})^+|^2dx\leq
\int_{\Sigma_{\lambda_0+\varepsilon}}(|\nabla u|+|\nabla u_{\lambda_0+\varepsilon}|)^{p-2}|\nabla(u-u_{\lambda_0+\varepsilon})^+|^2dx
\end{equation}
that is a contradiction unless $\int_{\Sigma_{\lambda_0+\varepsilon}}(|\nabla u|+|\nabla u_{\lambda_0+\varepsilon}|)^{p-2}|\nabla(u-u_{\lambda_0+\varepsilon})^+|^2dx=0$. But in this case it occurs that $(u-u_{\lambda_0+\varepsilon})^+=0$
for any $0\leq\varepsilon\leq\bar\varepsilon$, that is a contradiction with the definition of $\lambda_0$. Consequently we have proved that necessarily
\[
u=u_{\lambda_0}\qquad\text{in}\quad\Sigma_{\lambda_0}.
\]
Furthermore $u<u_{\lambda}$ in $\Sigma_{\lambda}$ for any $\lambda<\lambda_0$ and consequently $u$ is monotone increasing in $\Sigma_{\lambda_0}$ so that $u_{x_1}\geq 0$ in $\Sigma_{\lambda_0}$. By the strong maximum principle for the linearized operator \cite[Theorem 1.2]{DS2} it follows that actually
$u_{x_1}> 0$ in $\Sigma_{\lambda_0}$.\\

\noindent \emph{Step 3.}
{ \it Conclusion.} \\
The symmetry result follows now in a standard way, see e.g. \cite{AdvS,DPR,GNN,GNN2},
considering  the $N$
linearly independent directions $e^i$ with $i=1,\ldots,N$ in $\R^N$.  Exploiting the moving plane technique exactly as above it follows in fact that
 $u$ is  symmetric about some point $x_0= \cap _{i=1}^N T_{\lambda _0 (e_i)}^{e_i}$
 which is the only critical point of $u$. Furthermore, by the moving plane procedure exploited in any direction $\nu\in S^{N-1}$ we finally get that
 $u$ is radial and radially decreasing. Assuming up to translations that $x_0=0$ we get that $u=u(r)$ with
  $u'(r)<0$ for $r>0$.

\end{proof}

\begin{proof}[\underline{Proof of Theorem \ref{radialtheorem}}]
Once that the radial symmetry of the solutions has been proved in Theorem \ref{hfgdhgdhghfghdkgfvks},
the proof of Theorem \ref{radialtheorem} follows directly by the classification of the radial solutions obtained  in \cite{bidaut,ClassGV}.
\end{proof}

\end{document}